\documentclass[12pt,reqno]{amsart}
\usepackage[utf8]{inputenc}
\usepackage[dvipsnames]{xcolor}
\usepackage{tikz}
\usepackage{stmaryrd}
\usetikzlibrary{matrix,arrows,decorations.pathmorphing}
\usepackage{graphicx}
\usepackage{amssymb}
 \usepackage[stable]{footmisc}
\usepackage{amsmath,mathtools}
\usepackage{amsthm}
\usepackage{fullpage}
\usepackage{caption}
\usepackage{mathrsfs}
\usepackage{thmtools}
\usepackage{blkarray}
\usepackage{multirow}
\usepackage{picinpar} 
\usepackage{tikz-cd}
\usepackage{color}
\usepackage{verbatim}
\usepackage{hyperref}
\hypersetup{colorlinks=true, linkcolor=black, citecolor = OliveGreen, urlcolor = black}
\usepackage{amssymb}
\usepackage{amsmath}
\usepackage[shortlabels]{enumitem}
\usepackage{colonequals}
\usepackage{faktor}
\usepackage{color}
\usepackage{mathrsfs}
\usepackage[all]{xy}
\usepackage{comment}
\usepackage{mathpazo}
\usepackage{euler}
\usepackage[noabbrev,capitalise]{cleveref}
\usepackage[T1]{fontenc}
\usepackage[margin=1.in]{geometry}
\tikzset{
  closed/.style = {decoration = {markings, mark = at position 0.5 with { \node[transform shape, xscale = .8, yscale=.4] {/}; } }, postaction = {decorate} },
  open/.style = {decoration = {markings, mark = at position 0.5 with { \node[transform shape, scale = .7] {$\circ$}; } }, postaction = {decorate} }
}

\makeatother

\newtheorem{lemma}{Lemma}
\newtheorem{proposition}{Proposition}
\newtheorem{corollary}{Corollary}
\newtheorem{definition}{Definition}

\newtheorem{remark}{Remark}
\newtheorem{theorem}{Theorem}

\newtheorem{question}{Question}

\newtheorem*{claim*}{Claim}

\newtheorem{manualtheoreminner}{Theorem}
\newenvironment{manualtheorem}[1]{%
  \renewcommand\themanualtheoreminner{#1}%
  \manualtheoreminner
}{\endmanualtheoreminner}

\newtheorem{manualpropositioninner}{Proposition}
\newenvironment{manualproposition}[1]{%
  \renewcommand\themanualpropositioninner{#1}%
  \manualpropositioninner
}{\endmanualpropositioninner}

\newtheorem{manualremarkinner}{Remark}
\newenvironment{manualremark}[1]{%
  \renewcommand\themanualremarkinner{#1}%
  \manualremarkinner
}{\endmanualremarkinner}

\def\Spec{\operatorname{Spec}}

\title{Remarks on a theorem of Silverman}
\author{Khai-Hoan Nguyen-Dang\\
Quang-Khai Nguyen}
\date{\today}
\begin{document}

\begin{abstract}
   Motivated by a theorem of Silverman, we consider the following problem. Let \(A\) be an abelian variety over a global field \(K\). Given a non-torsion point \(P \in A(K)\), for a sufficiently large positive integer \(n\), whether there exists a place \(v\) of \(K\) such that the order of the reduction of \(P\) modulo \(v\) is \(n\)? In this article, we first show that this holds for an elliptic curve over a global function field of positive characteristic \(p>3\) and for sufficiently large positive integers \(n\) coprime to \(p\). In the second part of the paper, we consider its relative version over \(\mathbb C\). More precisely, let \(\pi: \mathcal{A} \rightarrow S\) be an abelian scheme over some variety \(S\) over \(\mathbb C\), and let \(P\) be a non-torsion section of \(\pi\). If the Betti map associated to \(P\) is generically submersive, then for every sufficiently large \(n\), there is a point \(s\) in \(S(\mathbb C)\) such that \(P(s)\) is a point of order \(n\) in the corresponding fiber.
\end{abstract}
\maketitle
\tableofcontents

\section{Introduction}
In the 19th century, Bang \cite{Ban86} and Zsigmondy \cite{Zsi92} demonstrated the existence of a constant \(M > 0\) so that for any nonzero rational number \(x\) with \(x \neq \pm 1\) and every integer \(n > M\), there exists a prime ideal \(\mathfrak{p}\) of \(\mathbb Z\) such that the order of \(x\) modulo \(\mathfrak{p}\) equals \(n\). This foundational result has since been generalized by various mathematicians and has found numerous applications. Notably, Schinzel \cite{Sch74} extended this theorem to number fields. He established that for any number field \(K\), there exists a constant \(n(K)\) dependent only on \(K\) so that for every \(x \in K^\times\) that is not a root of unity and every integer \(n > n(K)\), there exists a prime ideal \(\mathfrak{p}\) of \(\mathcal{O}_K\) such that the order of \(x\) modulo \(\mathfrak{p}\) is \(n\).


It is then natural to study these results in the setting of algebraic groups. Let $G$ be an algebraic group over a number field $K$. By spreading out, we can find an integral model $\mathcal{G}$ of $G$ over some open subscheme $\mathcal U$ of $\Spec \mathcal{O}_K$ such that $\mathcal{G}$ is a $\mathcal U$-group scheme. For a closed point $v\in \mathcal U$ (or equivalently, a \emph{place} of $K$), denote by  $\kappa_v$ the residue field at $v$. Then, we have the \emph{reduction homomorphism} at $v$, given by 
$$\mathrm{red}_v:\mathcal G(\mathcal{U})\to \mathcal G(\kappa_v), P\mapsto P\bmod v.$$
For $P\in \mathcal G(\mathcal U)$, we denote by $\mathrm{ord}_v(P)$ the order of $\mathrm{red}_v(P)$ in the finite group $\mathcal{G}(\kappa_v)$. For a complete account of reductions and integral models, we refer to Perucca's thesis \cite{Per08}. Now, we observe that a necessary condition for the existence of $v\in U$ such that $\mathrm{ord}_v(P)=n$ for every integer $n$ large is that $P$ is a non-torsion point. In the spirit of the results of Bang, Zsigmondy and Schinzel, one expects that this is also a sufficient condition.

For an elliptic curve \(E\) over $\mathbb Q$ and a non-torsion point \(P\in E(\mathbb{Q})\), Silverman \cite{Sil88} shows that for every sufficiently large positive integer \(n\), there exists a place $v$ such that \(\mathrm{ord}_v(P)=n\) by using \emph{height estimates} and \emph{Siegel's theorem}. In  \cite{CH99}, Cheon--Hahn extend Silverman's theorem to elliptic curves over any number field. The result is stated as follows.

\begin{theorem}(Silverman, Cheon--Hahn)\label{Theorem: Cheon-Hahn}
    Let $E$ be an elliptic curve over a number field and 
$P \in E(K)$ be a non-torsion point. Then for every sufficiently large integer $n$, there exists a place $v$ of good reduction so that the order of $P$  mod $v$ equals  $n$. Moreover, for all but finitely many $P$, there exists such  $v$ for each   $n>0$.
\end{theorem}

The proof of this result is not effective. In \cite{Ver23}, Verzobio established an effective bound for \(n\) building on the works of David \cite{Dav95} and Stange \cite{Sta16}.

\begin{remark}\label{remark: failure in higher dimensions}
    We can easily find that the above problem is not true in general by considering \(A\) a not geometrically simple abelian variety, say \(A=A_1 \times A_2\), and \(p=(p_1,p_2)\) where \(p_1\) is non-torsion in \(A_1\) and \(p_2\) is torsion in \(A_2\).
\end{remark}

These contributions lead to the following question.

\begin{question}\label{question1}
    Let \(G\) be an algebraic group over a number field \(K\). Given a non-torsion point \(P \in G(K)\), for a sufficiently large positive integer \(n\), whether (or when) there exists a place $v$ of $K$ such that $\mathrm{ord}_v(P)=n$.
\end{question}
To the best of our knowledge, the problem remains unresolved for higher-dimensional cases. A first step toward higher-dimensional abelian varieties is studied for a power of an elliptic curve in \cite{BNV23} by Bara\'nczuk-Naskr\k{e}ck-Verzobio.

We can describe the difficulty in higher-dimensional cases as follows. Roughly speaking, in the elliptic curve case, the condition that $\mathrm{red}_v(nP)$ is the identity is detected by a divisor (the zero section), hence by a denominator ideal. In higher dimensions, the condition that $\mathrm{red}_v(nP)$ is the identity is naturally encoded by the pullback of the ideal sheaf of the identity section (that is, integrality with respect to a point).

More precisely, let $A$ be an abelian variety over a number field $K$ and let $S$ be a finite set of places (containing all the archimedean ones and all the places of bad reduction). Let $\mathcal{A}/\mathcal{O}_{K,S}$ be the N\'eron model over the ring of $S-$integers $\mathcal{O}_{K,S}$  of $A/K$, and let
\[
e:\Spec(\mathcal{O}_{K,S})\longrightarrow \mathcal{A}
\]
    be the identity section.  Since $\mathcal{A}\to \Spec(\mathcal{O}_{K,S})$ is
    separated, $e$ is a closed immersion \cite[Tag~024T]{stacks-project}; write $\mathcal{I}_e\subset \mathcal{O}_{\mathcal{A}}$ for its ideal sheaf.

    For a $K$-point $P\in A(K)$, the N\'eron mapping property gives a unique section $\sigma_P:\Spec(\mathcal{O}_{K,S})\to \mathcal{A}$ extending $P$. We define the \emph{(scheme-theoretic) denominator ideal of $P$ relative to the origin} by
\[
\mathfrak{a}_S(P)\ :=\ \sigma_P^{\ast}(\mathcal{I}_e)\ \subset\ \mathcal{O}_{K,S}.
\]
Then for every non-archimedean place $v\notin S$, one has
\[
v\big(\mathfrak{a}_S(P)\big)>0
\quad\Longleftrightarrow\quad
\mathrm{red}_v(P)=e(v),
\]
i.e.\ $\mathrm{Supp}(\mathfrak{a}_S(P))$ is precisely the set of places where $P$
specializes to the identity section (equivalently, where the two sections
$\sigma_P$ and $e$ coincide on the special fibre).

We will say that $P$ is \emph{$S$-integral with respect to the origin} if
$\mathfrak{a}_S(P)=\mathcal{O}_{K,S}$, i.e., if $\mathrm{red}_v(P)\neq e(v)$ for all
$v\notin S$.

In the elliptic curve case ($\dim A=1$), the image $e(\Spec\mathcal{O}_{K,S})$
has codimension $1$ in $\mathcal{A}$; moreover $e$ is a regular immersion because it
is a section of a smooth morphism \cite[Tag~067R]{stacks-project}, hence
$e(\Spec\mathcal{O}_{K,S})$ is an effective Cartier divisor
\cite[Tag~01WQ]{stacks-project}.  Consequently $\mathcal{I}_e$ is invertible, and
$\mathfrak{a}_S(nP)$ recovers (up to the usual harmless $S$-factors) the familiar
denominator ideals/elliptic divisibility sequence.

By contrast, if $\dim A=g\ge 2$, then the origin on the generic fibre has
codimension $g\ge 2$ and so it is \emph{not} a (Weil) divisor (prime divisors are,
by definition, codimension~$1$) \cite[Tag~0BE0]{stacks-project}.  Thus the condition
``$nP$ reduces to the origin'' is naturally an \emph{integrality condition with
respect to a higher-codimension closed subscheme}, rather than a divisor condition. One can nevertheless attach local ``Weil functions'' and height functions to general ideal sheaves (equivalently, closed subschemes), extending
the classical divisor case; see e.g. \cite[\S2.2, Theorem~2.1]{Sil87} and \cite[\S2]{Voj12}.

Finally, the higher-codimension nature explains why this kind of integrality is typically much less restrictive than divisor-integrality: for instance, on a normal local domain $(R,\mathfrak{m})$ of dimension $\ge 2$, removing the closed point does not change global functions on the punctured spectrum (\emph{Hartogs phenomenon}) \cite[Tag~0BCS]{stacks-project}.
\subsection{Elliptic curves over global function fields}

In this note, we first prove Silverman's theorem for global function fields of positive characteristic. 

Throughout, $p$ is a prime number and $\mathbb{F}_q$ denotes the field of $q$ elements with $q$ is q power of $p$. Let $\mathcal{C}$ be a smooth projective geometrically irreducible curve over $\mathbb{F}_q$, and $K=\mathbb{F}_q(\mathcal{C})$ be the function field of $\mathcal{C}$. Then $K$ is called a \emph{global function field}. We denote by $M_K$ the set of \emph{places} (i.e., equivalence classes of valuations) of $K$, or equivalently, the set of closed points of $\mathcal{C}$. 

Let $G$ be an algebraic group over $K$. Using the spreading-out principle again, there is a small affine curve $\mathcal U\subset \mathcal{C}$ and an integral model group scheme $\mathcal{G}$ of $G$ over  $\mathcal{U}$. For a closed point $v\in \mathcal{U}$ (a place of $K$), we have the \emph{reduction homomorphism} at $v$, given by
$$\mathrm{red}_v:\mathcal G(\mathcal{U})\to \mathcal G(\kappa_v),  P\mapsto P\bmod v.$$
 For $P\in \mathcal G(\mathcal U)$, we also denote by $\mathrm{ord}_v(P)$ the order of $\mathrm{red}_v(P)$ in the finite group $\mathcal G(\kappa_v)$.
One can also ask Question \ref{question1} in this setting. When $G=E$ is an elliptic curve, we obtain the following result.
\begin{manualtheorem}{A}(= Theorem~\ref{Elliptic curve})\label{Theorem A}
    Let $E$ be an elliptic curve over some global function field $K$ of characteristic $p>3$ and let
$P \in E(K)$ be a non-torsion point. Then for every sufficiently large integer $n$ coprime to $p$, there exists a place $v$ of good reduction so that the order of $P$ mod $v$ is equal to $n$. 
\end{manualtheorem}

The problem considered in Theorem~\ref{Theorem A} is closely related to the divisibility sequence. The idea is as follows. Let $E/K$ be an elliptic curve, and let
$\pi:\mathcal{E}\to C$ be its N\'eron model with zero section $\mathcal{O}$.
For $P\in E(K)$, we write $\sigma_P:C\to\mathcal{E}$ for the corresponding section and define
\[
D_{nP}:=\sigma_{nP}^{\ast}(\mathcal{O})\in \mathrm{Div}(C)
\qquad (n\ge 1),
\]
as in the function-field theory of elliptic divisibility sequences (EDS) \cite{IMS12,Nas16,NS20}
(cf.\ also the valuation-theoretic denominator definition \cite[eq.\ (1.2)]{NS20}).
A place $v$ of $K$ is a \emph{primitive valuation} of $D_{nP}$ if $v\in\mathrm{Supp}(D_{nP})$ but
$v\notin\mathrm{Supp}(D_{mP})$ for every $1\le m<n$.
If $v$ is a place of \emph{good} reduction for $E$, then $v\in\mathrm{Supp}(D_{nP})$ if and only if
$\mathrm{red}_v(nP)= \mathcal{O}$ in $E(\kappa_v)$, and the minimal such $n$ is precisely the
order of the reduction $\mathrm{ord}_v(P)$.
Hence, for $v$ at which $E$ has good reduction, we have
\[
\mathrm{ord}_v(P)=n
\quad\Longleftrightarrow\quad
v \text{ is a primitive valuation of } D_{nP}.
\]
In particular, Theorem~\ref{Elliptic curve} may be read as an ``elliptic Zsigmondy'' statement
(over global function fields) for \emph{good} places and indices $\gcd(n,p)=1$.
This is the same equivalence between primitive divisors/valuations and exact reduction orders
highlighted explicitly in \cite[\S 1]{NS20}.


The existence of primitive valuations in EDS over function fields is well established in several regimes:
in characteristic $0$ it is proved in \cite[Theorem~1.7]{IMS12} (see also \cite[Theorem~5.5]{IMS12});
in positive characteristic $p\ge 5$ it is extended to \emph{ordinary} elliptic curves by Naskrk\k{e}cki,
with explicit bounds and a tame/wild dichotomy (including a formulation over finite constant fields in
\cite[Theorem~2.4]{Nas16}); and in the \emph{constant $j$-invariant} case with an additional condition that $(E,P)$ is non-constant, Naskrk\k{e}cki--Streng obtain sharp Zsigmondy bounds and a detailed classification of the supersingular obstruction in \cite[\S 1.1]{NS20}. Via the equivalence above, these results imply Silverman-type exact reduction order statements in many overlapping cases. However, we note that due to the additional condition on $(E,P)$ in the constant $j$-invariant case, the results of \cite{NS20} do not cover all cases of supersingular elliptic curves with an arbitrary non-torsion point $P$; \emph{a fortiori}, they do not cover Theorem~\ref{Theorem A}.

Our contribution is primarily \emph{methodological and expository}: we give a direct Silverman--Cheon--Hahn style argument over \emph{global function fields} (over a finite constant field), phrased from the outset in terms of reduction orders and using only the height decomposition together with a formal-group input away from the characteristic (hence the restriction $\gcd(n,p)=1$). This avoids the elliptic-surface intersection-theoretic framework and the tame/wild analysis central to \cite{Nas16} and does not require the constant-$j$ reductions and component-group bookkeeping used in \cite{NS20}. While we do not aim for explicit Zsigmondy bounds (as in \cite{Nas16,NS20}),
the resulting proof is short, self-contained, and tailored to the ``order modulo $v$'' formulation that
interfaces naturally with our later relative questions in Section~\ref{sec:unlikely}.
\begin{remark}\label{Remark: Failure of Siegel's theorem}
    The second claim in Theorem \ref{Theorem: Cheon-Hahn} is not true in general in positive characteristic. The reason is that the proof over number fields required the following theorem of Siegel: Let $E$ be an elliptic curve over a number field $K$, then for every place $v\in M_K$, we have $\lim \limits _{\substack{%
    P\in E(K)\\
    h(P)\to\infty }}\dfrac{h_{x,v}(P)}{h_x(P)}=0$ (we refer to \cite{CH99} and \cite[Chapter~IX, Theorem~3.1]{Sil86} for related notations). The proof of Siegel's Theorem crucially relies on Roth's Theorem. However, Roth's Theorem in positive characteristic does not hold,  see \cite{Osg75,BS76} for some counterexamples. In fact, we will provide explicit examples of elliptic curves over $\mathbb{F}_p(t)$ that do not satisfy both Siegel's Theorem (see Proposition \ref{rem:local-global-fails}) and the second assertion of Theorem \ref{Theorem: Cheon-Hahn} (see Proposition \ref{prop:failure-all-but-finitely-many}).    
\end{remark}

\subsection{A relative Silverman theorem} In the second part of this note, we reformulate our question in a different framework, namely, into a family of abelian varieties of characteristic \(0\) as follows.

\begin{question} \label{question2}
    Let \(A \rightarrow S\) be an abelian scheme over some variety \(S\) over an algebraically closed field \(L\) of characteristic \(0\), and let \(P: S \rightarrow A\) be a non-torsion section. When is it true that for every sufficiently large \(N\), there is a point \(s\) in \(S(L)\) such that \(P(s)\) is a point of order \(N\) in the corresponding fiber?
\end{question}

We can address this question by using the Betti map introduced by Andr\'e--Corvaja--Zannier \cite{ACZ20}.
\begin{manualtheorem}{B} (= Theorem~\ref{gen-sub})
    Let \(\pi: \mathcal{A} \rightarrow S\) be an abelian scheme over a smooth irreducible quasi-projective variety $S/\mathbb{C}$ of relative dimension $g\ge 1$, and let \(P\) be a non-torsion section of \(\pi\). If the Betti map associated to \(P\) is generically submersive (in the sense of \cite{ACZ20}), then for every sufficiently large \(N\), there is a point \(s\) in \(S(\mathbb C)\) such that \(P(s)\) is a point of order \(N\) in the corresponding fiber.
\end{manualtheorem}

The theorem applies in the concrete situations where generic submersivity of the Betti map is known, e.g.\ those established in \cite{ACZ20,Gao20} (see also the summary in \cite[\S 1.1]{GH23}). For example, we can take \(\mathcal A/S\) with \(\mathrm{End}(\mathcal A/S) = \mathbb{Z}\) and \(\dim S \geq g\) \cite[Theorem 6.1.1]{ACZ20}. Results of this nature for certain elliptic surfaces (i.e.\ $g=1$ over a curve base) appear implicitly in the work of Ulmer--Urz\'ua \cite{UU21}, and related density results are proved in \cite{CDMZ22}.

Let \(X:= P(S)\), we can transform the question \ref{question2} to the following question.

\begin{question}
    Let $\pi:\mathcal A\to S$ be an abelian scheme and let $P:S\to\mathcal A$ be a section. Set $X:=P(S)$. Is $X(\mathbb C)\cap\mathcal A_{\mathrm{tor}}$ Zariski dense in $X$?
\end{question}

Drawing on recent advancements in the field of unlikely intersections, we establish the following criteria. It is a corollary of the relative Manin--Mumford conjecture established by Gao--Habegger \cite[Theorem~1.3]{GH23}. 
\begin{corollary}(= Corollary \ref{cor:GH23-to-silverman})
  Let $\pi:\mathcal A\to S$ be an abelian scheme over $\mathbb C$, $P:S\to\mathcal A$ be a section, and set $X:=P(S)$. Assume that \(\mathbb{Z}X\) is Zariski dense in \(\mathcal{A}\). Then the question \ref{question2} holds if $X(\mathbb{C})\cap \mathcal A_{\mathrm{tor}}$ is Zariski dense in $X$.
\end{corollary}

\subsection*{Another reformulation}
We end the introduction with the following interpretation of the problem. Let \(A\) be an abelian variety defined over a number field \(K\) with the identity $\mathcal{O}$, and let \(P\) be a non-torsion point in \(A(K)\). We denote the ring of integers of \(K\) by \(\mathcal{O}_K\) and the Néron model for \(A/K\) by \(\mathcal{A}/\mathcal{O}_K\). Consider a finite set of places \(S\) in \(K\). For each \(n \geq 1\), we define the integral ideal \(C_n(\mathcal{A},P,S)\) in \(\mathcal{O}_K\) as 
\[
C_n(\mathcal{A},P,S) = \prod_{\substack{\mathfrak{p}\in\Spec\mathcal{O}_K \setminus S:\\nP \equiv \mathcal{O} \bmod \mathfrak{p}}} \mathfrak{p},
\]
where \(nP \equiv \mathcal{O} \bmod\mathfrak{p}\) signifies that \(nP\) reduces to the identity in \(\mathcal{A}(\kappa_\mathfrak{p})\). In the special case where \(A\) is an elliptic curve, \(P\) is a point on it, and \(\mathcal{A}\) is described by its Weierstrass equation, each term \(C_n(A,P,S)\) represents the product of prime ideals not in \(S\) that divide the \(n\)-th term of the corresponding elliptic divisibility sequence \(B_n(E,P)\). This sequence is determined by the denominators of the \(x\)-coordinates of the point \(nP\). Under this reformulation, Question \ref{question1} transforms into finding a primitive divisor of \(C_n(A,P,S)\) for all but finitely many \(n\). By \cite{Sil88,CH99}, we know an affirmative answer when $E$ is an elliptic curve. For divisibility sequences involving powers of elliptic curves, refer to \cite{BNV23}. For the geometric counterpart of this notion, consult \cite{BCZ22,BCT24}. 

\subsection*{Organization of the paper}
 In Section 2, we establish a version of Silverman’s theorem for elliptic curves over global function fields, leveraging height machinery and reductions of algebraic groups. Section 3 recalls recent progress in the realm of unlikely intersections, and we then set forth criteria for addressing the Silverman problem for a specific family of abelian varieties. We present a proof of Schinzel’s theorem over global function fields in the appendix.

\subsection*{Acknowledgements}
This article was conducted when the second author was a master's student at the Vietnam Institute of Mathematics. We extend our sincere gratitude to Professor Quoc-Thang Nguyen for his supervision and support throughout this project. Additionally, we deeply appreciate the invaluable guidance provided by Professors Umberto Zannier, Fabrizio Barroero, and Ziyang Gao. We thank Professor Yann Bugeaud for pointing out that Roth's theorem does not hold in positive characteristic. We are grateful to the anonymous referee for a careful reading and for a detailed, constructive report, including numerous valuable suggestions, corrections, and guidance that substantially improved the paper.

\section{The Silverman theorem for elliptic curves over global function fields}\label{Section2}
In this section, we prove a version of Silverman's theorem for elliptic curves over global function fields by using height machinery and reduction of elliptic curves. 

\subsection{Preparations}


Let $K=\mathbb F_q(\mathcal{C})$ be the function field of a smooth projective geometrically irreducible curve over $\mathbb F_q$ in characteristic $p>3$. We fix an algebraic closure $\overline{K}$.  For each $v\in M_K$, the valuation $v(x)$ is given by the order of vanishing of $x\in K$ at $v$. We write $\mathcal{O}_v$ for the local ring at $v$ (the ring
of rational functions on $\mathcal C$ regular at $v$), $\kappa_v$ for the residue field at $v$, and $\deg(v)=[\kappa_v:\mathbb F_q]$ for the degree of $v$. We also write $\mathrm{red}_v$ for the reduction homomorphism $\mathcal{O}_v\twoheadrightarrow  \kappa_v,x\mapsto x\bmod v$. 


For $x\in K^\times$, we have  $v(x)=0$ for all but finitely many $v\in M_K$, and the product formula reads as $\sum_{v\in M_K}v(x)\deg(v)=0$, cf. \cite[Chapter~3]{Lang-1983}. The \emph{logarithmic height} on $\mathbb{P}^n(K)$ is defined as 
$$h(\alpha_0:\ldots:\alpha_n)=\sum_{v\in M_K}-\min\{v(\alpha_0),\ldots,v(\alpha_n)\}\deg(v).$$

Let $E$ be an elliptic curve over $K$ with the identity element $\mathcal O$. We may and do assume that  $E$  is given by a Weierstrass equation $y^2=x^3+ax+b\text{ } (a,b\in K)$ that is \emph{minimal} outside some finite set of places $S$. In particular, $a$ and $b$ belong to $\mathcal{O}_v$ for all $v\not\in S$. For convenience, we write $\mathcal{O}$ for $\mathrm{red}_v(\mathcal{O})$ for all $v\in M_K$. We write $[n]:E\to E$ for the multiplication-by-$n$ morphism and write $nP$ for $[n](P)$.


 For $P=(x,y)\in E(K)$, following \cite[p.~38]{Zim76} we consider the following naive local distribution $h_v$:  $$ h_v(P):=\left\{
	\begin{array}{ll}
		-\min\{0,v(x),v(y)\} \deg(v) & \mbox{if } P \neq \mathcal O \\
		 0& \mbox{if } P=\mathcal O
	\end{array}
\right.$$
Then $h(P)=\sum_{v\in M_K}h_v(P)$ where we view $P=[x:y:1]$ as a point in $\mathbb P^3(K)$. 

For each rational function $f\in K(E)$,  the height relative to $f$ is the function $h_f:E({K})\to \mathbb R$ defined by $h_f(P)=h(f(P)).$ We define the  \emph{N\'eron--Tate height} $\hat{h}$ on $E({K})$ by the Tate limit: 
$$\hat{h}(P)=\lim_{N\to\infty}\dfrac{h_f(2^NP)}{4^N\deg f}$$ where $f$ is any non-constant even function. The existence of $\hat{h}$ and its independence of $f$ follow precisely \cite[Chapter~VIII, Proposition~9.1]{Sil86}, see also \cite[p.~39]{Zim76}. 
\begin{proposition}\label{Proposition: Neron-Tate height}The N\'eron--Tate height satisfies the following properties.
    \begin{enumerate}
        \item The Northcott property (for global function fields with finite constant field): for any real number $B$, the set 
        $$\{P\in E({K}):\hat{h}(P)<B\}$$ is finite.
        \item For any integer $m$ and $P\in E({K})$, $\hat{h}(mP)=m^2\hat{h}(P)$.
        \item A point $P\in E(K)$ is torsion if and only if $\hat{h}(P)=0$.
        \item (\cite[p.~40]{Zim76}) There exists a constant $c>0$ depending only on $E/K$ such that $|h(P)-\hat{h}(P)|<c$ for all $P\in E(K)$.
    \end{enumerate}
\end{proposition}

To prove Theorem~\ref{Theorem A}, we adapt the strategy used in the proof of the main theorem of \cite{CH99}. The following lemma is standard; we present a proof here for the reader's convenience. More general results, with $n$ not necessarily coprime to $p$, can be found in \cite[Lemma~8.2]{Nas16} (for ordinary elliptic curves) and in \cite{Naskrecki-Verzobio-2025} for general versions that encompass both good and bad reduction places.
\begin{lemma}\label{detecting torsion points}
Let $v\in M_K\setminus S$. Let $P$ be a non-torsion point of $E(K)$. Then 
\begin{enumerate}
    \item  If $\mathrm{red}_v(P)\neq\mathcal{O}$, we have $h_v(P)=0$.
    \item If $\mathrm{red}_v(P)=\mathcal{O}$, we have $$h_v(nP)=h_v(P)>0$$  for any positive integer $n$ coprime to $p$. 
\end{enumerate} 
\end{lemma}
\begin{proof}
We write $P=(x,y)$ and $P=[X:Y:Z]$ in the corresponding projective closure of $E$ in $\mathbb P^2$, where (after scaling) $X,Y,Z\in \mathcal O_v$.
The condition  $\mathrm{red}_v(P)=\mathcal{O}$ means that $v(X)>v(Y)$, $v(Z)>v(Y)$, and whence $v(y)<0.$ Therefore, the condition $\mathrm{red}_v(P)\neq\mathcal{O}$ is equivalent to either $v(X)\leq v(Y)$ or $v(Z)\leq v(Y)$. If $v(X)\leq v(Y)$ and $v(Y)<v(Z)$, then $v(X)<v(Z)$.  But then from the homogeneous Weierstrass equation $Y^2Z=X^3+aXZ^2+bZ^3$ we obtain $$2v(Y)+v(Z)=3 v(X),$$a contradiction.  If $v(Z)\leq v(Y)$, then $v(y)\geq0$, and from $y^2=x^3+ax+b$ we obtain $v(x)\geq 0$ (since if $v(x)<0$, then $2v(y)=v(x^3+ax+b)=3v(x)<0$, a contradiction), i.e., $h_v(P)=0$ as wanted. 

For the second statement, since the condition $\mathrm{red}_v(P)=\mathcal{O}$ yields that $v(y)<0$, we deduce from $y^2=x^3+ax+b$ that $v(x)<0$ and  $3v(x)=2v(y),$ whence $h_v(P)=-v(y)>0.$ We denote $E_1(K_v):=\{M\in E(K_v):P\bmod \mathfrak{m}_v=\mathcal{O}\}, $
then we have an isomorphism of groups (see \cite[Chapter~VII, Proposition~2.2]{Sil86})
$$E_1(K_v)\to \hat{E}(\mathfrak m_v) ,\text{ }M=(x(M),y(M))\mapsto z(M) $$ where we write $K_v$ for the completion of $K$ at $v$, $\mathcal{O}_{K_v}$ for the ring of integers of $K_v$, $\mathfrak m_v$ for the maximal ideal of $\mathcal{O}_{K_v}$,  $\hat{E}$ for the formal group over $\mathcal{O}_{K_v}$ associated to $E$ at $v$, and $z=-x/y$ for a parameter at $\mathcal{O}$. This isomorphism gives us $v(y(M))=-3v(z(M))$ for all $M\in E_1(K_v)$. Via this isomorphism, $nP$ maps to $[n](z)=nz+O(z^2)\in\mathcal{O}_{K_v}\llbracket z \rrbracket$ where $[n]$ denotes the multiplication-by-$n$ on $\hat{E}$. Since $v(z)=v(x)-v(y)>0$ and $v(n)=0$, $$ v(z(nP))=v\Big([n]z(P)\Big)\Big)=v\Big(z(P)\Big)=v(x)-v(y).$$ Consequently, we get $$v(y(nP))=-3v(z(nP))=3v(y)-3v(x)=3v(y)-2v(y)=v(y)<0.$$ So $3v(x(nP))=2v(y(nP))<0$, and $h_v(nP)=-v(y(nP))$, hence $h_v(nP)=h_v(P)>0$.
\end{proof}
\begin{remark}
    It follows from the proof of Lemma \ref{detecting torsion points} that for all $v\not\in S$ and  for all $P\in E(K)$,  $h_v(P)=\frac{-3}{2}\min\{0,v(x(P))\}\deg(v)$. Therefore, for all but finitely many $v$, the function $h_v$ coincides, up to a multiplicative constant, with the standard local height function $\frac{-1}{2}\min\{0, v(x)\}$ in \cite[Chapter~VI, \S2]{Sil94}.
\end{remark}
For general $v$, although Lemma \ref{detecting torsion points} does not hold, we still have the following boundedness result for $h_v(nP)$. 
\begin{lemma}\label{Lemma: weak Siegel} 
Let $E/K$ be an elliptic curve over a global function field of characteristic $p\neq 2$ and 3, and let $P\in E(K)$ be a non-torsion point.  Then for any place $v\in M_K$, we  have
 $h_v(nP)$ is bounded above when $n$ ranges over integers coprime to $p$. In particular, we have
$$\lim\limits _{\substack{%
  \gcd(n,p)=1  \\n\to\infty
     }}\dfrac{h_v(nP)}{h(nP)}=0.$$
\end{lemma}

\begin{proof}
Since the goal is to obtain the boundedness, replacing $x$ by $u^2x$ for some $u\in K^\times$, we may assume that the given Weierstrass equation is minimal at $v$. The desired boundedness then follows from the following observation: if there are $n\neq m$ and $n,m$ coprime to $p$ such that both $nP$ and $mP$ mod $\mathfrak p_v$ equal $\mathcal{O}$, then 
$h_v(nP)>0$, $h_v(mP)>0$ and $h_v(nP)=h_v(nmP)=h_v(mP)$ by Lemma \ref{detecting torsion points}. 


Since $h(nP)$ tends to $\infty$ as $n$ tends to $\infty$ (thanks to the Northcott property), the second claim follows.
\end{proof}
\begin{remark} We will see in Proposition \ref{rem:local-global-fails} that the limit $\displaystyle\lim_{n\to\infty} \dfrac{h_v(nP)}{h(nP)}$ might not exist and, in particular, need not equal 0 in general.
\end{remark}
\subsection{Proof of Theorem \ref{Theorem A}}
\begin{theorem}\label{Elliptic curve}
Let $E$ be an elliptic curve over a global function field $K$ of characteristic $p>3$
and let $P \in E(K)$ be a non-torsion point. Then for every sufficiently large integer $n$ coprime to $p$, there exists a place $v$ of good reduction such that the order of $P$ mod $v$ is $n$. 
\end{theorem}
\begin{proof}
By enlarging $S$ if necessary, we may assume that $E$ has good reduction outside $S$. We denote by $\# S$ the cardinality of the finite set $S$. 

Assume that for all sufficiently large integers $n>1$ with $\gcd(n,p)=1$, there is no place $v\notin S$ of good reduction such that $\mathrm{ord}_v(P)=n$. In other words, if $nP$ mod $v$ equals $\mathcal O$, then there exists some prime divisor $r$ of $n$ such that $\frac{n}{r}P$  mod $v$ equals $\mathcal O$, whence Lemma \ref{detecting torsion points} gives us $$h_v(nP)=h_v\Big(\dfrac{n}{r}P\Big).$$
    It follows that for $v\in M_{K}\setminus S$, we have \begin{equation}\label{1}
        h_v(nP)\leq \sum_{r}h_v\Big(\dfrac{n}{r}P\Big) ,
    \end{equation}
    where $r$ runs over the set of prime divisors of $n$. For $v\in S$, Lemma \ref{Lemma: weak Siegel} yields $$\lim_{\gcd(n,p)=1,n\to\infty}\dfrac{h_v(nP)}{h(nP)}=0.$$
    Since $\# S$ is finite, it follows that for any $\epsilon>0$,  
    \begin{equation}\label{2}
        h_v(nP)\leq\epsilon h(nP) 
    \end{equation}
    for all sufficiently large integers $n$. Combining the inequalities \eqref{1} and \eqref{2}, we get
    $$h(nP)=\sum_{v}h_v(nP) \leq\sum_{v\not\in S}\sum_{r|n}h_v\Big(\dfrac{n}{r}P\Big)+\sum_{v\in S}\epsilon. h(nP)\leq \sum_{r|n}h\Big(\dfrac{n}{r}P\Big)+\# S.\epsilon. h(nP).$$
    So \begin{equation}\label{eq1}
        (1-\# S.\epsilon)h(nP)\leq \sum_{r|n}h\Big(\dfrac{n}{r}P\Big). 
    \end{equation} 
We combine Proposition \ref{Proposition: Neron-Tate height}.(3) with \eqref{eq1} to deduce that
 $$(1-\# S.\epsilon)(\hat{h}(nP)-c) <\sum_{r|n}\hat{h}\Big(\dfrac{n}{r}P\Big)+c.n $$
 since $\#\{\text{prime divisors of }n\}< n$. Because of the quadraticity of $\hat{h}$, it follows that 
 $$ (1-\# S.\epsilon)(n^2.\hat{h}(P)-c)<\sum_{r|n}\dfrac{n^2}{r^2}\hat{h}(P)+c.n<\dfrac{n^2}{2}\hat{h}(P)+cn$$ since  
 \[
\sum_{r\mid n}\frac1{r^2}\ \le\ \frac14+\sum_{m\ge1}\frac1{(2m+1)^2}
\ =\ \frac14+\Bigl(\frac{\pi^2}{8}-1\Bigr)\ =\ \frac{\pi^2}{8}-\frac34\ <\ \frac12.
\]
 
 Therefore $$\Big(\dfrac{1}{2}-\# S.\epsilon\Big)n^2.\hat{h}(P)<(n+1-\# S.\epsilon).c $$ We choose $\epsilon<\frac{1}{2\# S}$ and let $n$ tend to $\infty$, we obtain $\hat{h}(P)=0$, a contradiction. 
 
\end{proof}
\begin{remark}
We note that the condition $\mathrm{gcd}(n,p)=1$ is necessary. Indeed, recall that an elliptic curve $E$ in characteristic $p>0$ is \emph{supersingular}, or that $E$ has \emph{Hasse invariant 0}, if $E[p]=0$, see \cite[Chapter~V, Theorem~3.1]{Sil86}. Otherwise, $E$ is said to be \emph{ordinary} or to have \emph{Hasse invariant} 1. Moreover, if $E$ is given by a homogeneous equation $f(X,Y,Z)=0$, then the Hasse invariant is 0 if and only if the coefficient of $(XYZ)^{p-1}$ in $f^{(p-1)}$ is 0, see \cite[Chapter~IV, Proposition~4.21]{Hartshorne-1977}. Now,  we consider a supersingular elliptic curve $E$ over $K$ with homogeneous equation $f=0$. For all but finitely many $v\in M_K$, the reduction of $E$ at $v$ is an elliptic curve $E_{v}$ over $\kappa_v$ given by $\mathrm{red}_v(f)=0$ where $\mathrm{red}_v(f)$ is the homogeneous polynomial in $\kappa_v[X,Y,Z]$ whose coefficients are the reduction of those of $f$ via $\mathrm{red}_v$. Thus, the coefficient of $(XYZ)^{p-1}$ in $\mathrm{red}_v(f)^{p-1}$ is also equal to 0, so $E_{v}$ is supersingular. It follows that for such $v$, $E_v[n]=E_v[np]$ for all integers $n$. Therefore, for all but finitely many $v$, the order $\mathrm{ord}_v(P)$ must be coprime to $p$. 
\end{remark}
When $E$ is ordinary, we have the following corollary. 
\begin{corollary}\label{ordinary elliptic curve}
    Let $E$ be an ordinary elliptic curve over some global function field $K$ of characteristic $p>3$ and let $P \in E(K)$ be a non-torsion point. Let $r$ be a positive integer. Then there exists $N(r)$ depending on $r$ such that for every $n>N(r)$ coprime to $p$, there exists a place $v$ of good reduction such that the order of $P$ mod $v$ is $np^r$.  
\end{corollary}
\begin{remark}
This corollary is weaker than \cite[Theorem~2.3 \& 2.4]{Nas16}, in which a uniform lower bound result for all $n$ (not necessarily coprime to $p$) that can be realized as reduction orders is given for all ordinary elliptic curves.
\end{remark}
\begin{proof}
Because $E$ is ordinary, there exists a point $Q\in E(\overline{K})$ of order $p^r$. We set $L:=K(E[p^r])$ the $p^r-$division field of $E$, and denote $E_L$ the base change of $E$ to $L$. Viewing $P$ as an $L-$ point of $E_L$, we observe that $P-Q\in E_L(L)$ is also a non-torsion point. We note the following properties of the reduction of points. 
\begin{enumerate}
    \item (\cite[Lemma~1.2.3]{Per08}) For all but finitely many places $v$ of $K$ the following holds:  if we write $w$ for any place of $L$ lying over $v$, then there is an inclusion $E(\kappa_v)\xhookrightarrow{}E_L(\kappa_w)$, so for every point in $E(K)$, the orders modulo $v$ and modulo $w$ agree.
    \item (\cite[Cor.~2.3.4]{Per08}) Since $Q$ is torsion,  $\mathrm{ord}_w(Q)$ equals the order of $Q$, which is $p^r$, for all but finitely many places $w$ of $L$.\footnote{This claim can also be proved by using the proof of Lemma \ref{detecting torsion points}, since for every $0\leq l<r$, there are only finitely many places $w$ such that $w(y(p^lQ))<0$.}

\end{enumerate}
We note that the proofs of those properties given
in \cite{Per08} also work well over global function fields. We call $V$ the set of  \emph{exceptional places} of $K$ in (1)
and  call $U$ the set of  \emph{exceptional places} of $L$ in (2)
and of places of $L$ lying over places in $V$. Both $V$ and $U$ are finite sets. 

Now, applying Theorem~\ref{Elliptic curve} for $P-Q\in E_L(L)$, there exists an integer $N(r)$ depending on $r$ such that for every integer $n>N(r)$ coprime to $p$, there exists a place $w\in M_{L}\setminus U$ of good reduction so that the order of $P-Q$ mod $w$ is $n$. In other words, we have $\mathrm{ord}(\mathrm{red}_w(P-Q))=n$,  $\mathrm{ord}(\mathrm{red}_w(Q))=p^r$ and $\mathrm{ord}(\mathrm{red}_w(P))=\mathrm{ord}(\mathrm{red}_v(P))$ where $v$ is the place in $M_K$ lying below $w$. The desired result then follows from the following group-theoretic result.
 \end{proof}
\begin{lemma}
  Let $n$ be a positive integer coprime to a prime number $p$. Let $P$ and $Q$ be two elements in a given commutative group such that $\mathrm{ord}(P-Q)=n$ and $\mathrm{ord}(Q)=p^r$. Then $\mathrm{ord}(P)=np^r$.
\end{lemma}
 \begin{proof}
 Let $\mathcal O$ be the identity element.      Since $np^rP=np^r(P-Q)+np^rQ=np^r(P-Q)=\mathcal O$, the order $\mathrm{ord}(P)$ must be of the form $mp^s$ where $m|n$ and $s\leq r$. Then $$\mathcal{O}=mp^rP=mp^rQ+mp^r(P-Q)=mp^r(P-Q), $$
hence $n|mp^r$ which implies that $m=n$. Similarly we have $$\mathcal{O}=np^sP=np^sQ+np^s(P-Q)=np^sQ. $$Thus $p^r|np^s$ which yields that $s=r$. Therefore, $\mathrm{ord}(P)=np^r$ as wanted.
 \end{proof}
 \begin{remark}
The analogue of Theorem~\ref{Elliptic curve}  fails for one-dimensional unipotent groups in characteristic $p$. Indeed, for any field $K$ of characteristic $p$, the group $U(K)$ of $K$-points of a
connected unipotent group $U$ is a $p$-group; in particular every element has
order a power of $p$ (and for $U=\mathbb G_a$ every nonzero element has order exactly $p$).
Hence, one cannot realize integers $n$ coprime to $p$ as reduction orders in this setting.
\end{remark}

\begin{remark}
As noticed in Remark \ref{remark: failure in higher dimensions}, for higher-dimensional cases, Theorem~\ref{Elliptic curve} does not hold for  (semi-)abelian varieties. For example, consider an elliptic curve $E_1$ over $K$ containing a non-torsion point $Q\in E_1(K)$ (see \cite{TS67} for infinitely many examples), and an elliptic curve $E_2$ that contains a torsion point $R\in E_2(K)$ of order $m>1$ coprime to $p$ (since we always have $E_2[m]\cong \mathbb Z/m\times\mathbb Z/m$, the existence of such $E_2$ is ensured by extending the constant field $\mathbb F_q$ to its suitable finite extension). Set $G:=E_1\times E_2$  and consider its non-torsion $K-$point $(Q,R)$. Now, if for any $n$ sufficiently large coprime to $p$, there is some $v\in M_K$ satisfying the order of $(Q,R)$ mod $v$ is $n$, then in particular, $\mathrm{ord}_v(R)=n$ for infinitely many $v$, which is absurd.
\end{remark}

\subsection{Failure of Siegel's theorem in positive characteristic}
This section is devoted to providing the examples mentioned in Remark \ref{Remark: Failure of Siegel's theorem} in the introduction.

While it is proved by Lang \cite[Chapter~7]{Lang-1983} that Roth's theorem is true for function fields in characteristic 0,  it was first noticed by Mahler \cite{Mahler-1949} that Roth's theorem does not hold for function fields in characteristic $p>0$ if considering algebraic elements of degree a power of $p$. It was shown further by Osgood \cite{Osg75} and Baum--Sweet \cite{BS76} that even after excluding such elements, Roth's theorem still does not hold in characteristic $p$. However, it is not clear from the literature whether Siegel's Theorem (see \cite{CH99} and \cite[Chapter~IX, Theorem~3.1]{Sil86}, formulated in terms of Roth-type height decay/local-to-global distribution), whose proof relies crucially on Roth's theorem,  holds in characteristic $p$ or not. By observing that Frobenius and purely inseparable isogenies can force $p$-power multiples of a non-torsion point to remain integral away from a fixed place, we produce concrete counterexamples to such a theorem of Siegel (Proposition \ref{rem:local-global-fails}) and to the second assertion of Theorem \ref{Theorem: Cheon-Hahn} (Proposition \ref{prop:failure-all-but-finitely-many}). 

We make the observation more explicitly in our setting. The following construction is extracted and repackaged from \cite[Example~9.3]{Nas16}.

\begin{lemma}[A supersingular isotrivial twist with integral $p$-power multiples] \label{lem:supersingular-isotrivial} Let $p\ge 5$ be a prime and let $K=\mathbb{F}_p(t)$. Denote by $v_\infty \in M_K$ the place at infinity, corresponding to the valuation of $1/t$. Choose a supersingular elliptic curve 
\[ 
E_0/\mathbb{F}_p:\quad Y^2=X^3+\alpha X+\beta, \qquad \alpha,\beta\in\mathbb{F}_p,\ \ \Delta(E_0)=-16(4\alpha^3+27\beta^2)\neq 0, 
\] 
and set 
\[ 
r(t):=t^3+\alpha t+\beta\in\mathbb{F}_p[t]. 
\] 
Define the (isotrivial) elliptic curve 
\[ 
E/K:\quad y^2=x^3+\alpha r(t)^2x+\beta r(t)^3 
\] 
and the $K$-rational point 
\[ R:=(t\,r(t),\,r(t)^2)\in E(K). \] 
Then for every $k\ge 0$, one has 
\[ 
p^kR=\big(t^{p^{2k}}\,r(t),\; (-1)^k\,r(t)^{(p^{2k}+3)/2}\big)\in E(K). 
\] 
In particular, $R$ is a non-torsion point. Moreover, for every finite place $v\neq v_\infty$ of $K$ and every $k\ge 0$, both affine coordinates of $p^kR$ lie in $\mathcal{O}_v$. 
\end{lemma}

\begin{proof}
Substituting $x=t r(t)$ and $y=r(t)^2$ into the defining equation of $E$ gives
\[
y^2=r(t)^4=(t^3+\alpha t+\beta)\,r(t)^3=x^3+\alpha r(t)^2x+\beta r(t)^3,
\]
so indeed $R\in E(K)$.

Let $L:=K(s)$ with $s^2=r(t)$.  Over $L$ the change of variables
\[
\varphi:E_L \xrightarrow{\ \sim\ } (E_0)_L,\qquad
(x,y)\longmapsto \Big(\frac{x}{r(t)},\,\frac{y}{s^3}\Big)
\]
is an isomorphism, with inverse $(X,Y)\mapsto (r(t)X,\,s^3Y)$.  Under $\varphi$, we have
\[
Q:=\varphi(R)=(t,s)\in E_0(L).
\]

Let $\pi$ denote the $p$-power Frobenius endomorphism of $E_0$, so
$\pi(X,Y)=(X^p,Y^p)$.  Since $E_0/\mathbb{F}_p$ is supersingular, one has
$a_p\equiv 0\pmod p$ where $a_p:=p+1-E_0(\mathbb F_p)$.  For $p\ge 5$ the Hasse bound $|a_p|\le 2\sqrt{p}<p$
forces $a_p=0$, so the characteristic polynomial of $\pi$ is $T^2+p$, hence
$\pi^2=[-p]$ in $\mathrm{End}(E_0)$.
Therefore
\[
[p]=[-1]\circ \pi^2
\qquad\text{and hence}\qquad
[p^k]=[-1]^k\circ \pi^{2k}\quad (k\ge 0).
\]
Applying this relation to $Q$ yields
\[
p^kQ=[-1]^k\big(t^{p^{2k}},\,s^{p^{2k}}\big)
=\big(t^{p^{2k}},\,(-1)^k s^{p^{2k}}\big)\in E_0(L).
\]
Transporting back via $\varphi^{-1}$ gives
\[
p^kR
=\Big(r(t)\,t^{p^{2k}},\ s^3\cdot (-1)^k s^{p^{2k}}\Big)
=\Big(t^{p^{2k}}r(t),\ (-1)^k s^{p^{2k}+3}\Big).
\]
Since $p^{2k}+3$ is even, we may write $p^{2k}+3=2\cdot \frac{p^{2k}+3}{2}$ and
thus $s^{p^{2k}+3}=(s^2)^{(p^{2k}+3)/2}=r(t)^{(p^{2k}+3)/2}\in K$, proving the
displayed formula. This is the same computation as in
\cite[Example~9.3]{Nas16}.

Finally, since $\deg(y(p^kR))=\frac{3(p^{2k}+3)}{2}\to\infty$ as $k\to\infty$, $R$ is non-torsion. If $v\neq v_\infty$ is a finite place of $K$, then $t\in\mathcal{O}_v$
and $r(t)\in\mathcal{O}_v$, so the explicit formula shows
$x(p^kR),y(p^kR)\in\mathcal{O}_v$ for all $k\ge 0$.
\end{proof}

The obstruction is genuinely positive-characteristic: for supersingular $E_0$ one has $[p]=[-1]\circ \pi^2$ with $\pi$ Frobenius, so multiplication-by-$p$ is purely inseparable. In the above isotrivial twist this forces the $p$-power multiples $[p^k]P$ to remain integral
away from the place at infinity. It follows that the local-to-global ratio in Siegel's Theorem \cite[Chapter~IX, Theorem~3.1]{Sil86} need not vanish in characteristic $p$.
\begin{proposition}
\label{rem:local-global-fails}
Let $p\ge 5$ and $K=\mathbb{F}_p(t)$.  Fix a place $v\in M_K$.  In general, for an elliptic curve $E$ over $K$, the
limit $\displaystyle\lim\limits _{\substack{%
  P\in E(K)  \\h(P)\to\infty
     }}\frac{h_v(P)}{h(P)}$ might not exist. In particular,
$\displaystyle\lim\limits _{\substack{%
  P\in E(K)  \\h(P)\to\infty
     }}\frac{h_v(P)}{h(P)}=0$
need not hold.
\end{proposition}

\begin{proof}
Let $E/K$ and $R\in E(K)$ be as in Lemma~\ref{lem:supersingular-isotrivial}, and
take $v=v_\infty$. We set $P_k:=p^kR$.

By Lemma~\ref{lem:supersingular-isotrivial}, for every $k\ge 0$ and every finite
place $w\neq v_\infty$, the point $P_k$ has $w$-integral coordinates, hence
$h_w(P_k)=0$.  Therefore
$
h(P_k)=h_{v_\infty}(P_k). $ Consequently,
\[
\frac{h_{v_\infty}(P_k)}{h(P_k)}=1 \qquad\text{for all }k\ge 1.
\]
Since $R$ is non-torsion, $h(P_k)\to\infty$ as $k\to\infty$ by the Northcott property. 
On the other hand, by Lemma \ref{Lemma: weak Siegel}, we know that $\displaystyle\lim\limits _{\substack{%
 \gcd(n,p)=1  \\n\to\infty}}
     \dfrac{h_v(nR)}{h(nR)}=0$. Therefore, the claimed limit might not exist.
\end{proof}
\begin{remark}
   In contrast, for function fields in characteristic 0, due to Manin and Voloch, $h_v(P)$ is always bounded above as $P$ ranges over $E(K)$, see \cite[Chapter~III, Theorem~12.2]{Sil94}. The proof crucially relies on the assumption that the base field has characteristic $0$.
\end{remark}
Consequently, the following assertion (the second claim in Theorem \ref{Theorem: Cheon-Hahn}) does not hold in characteristic $p>0$: for all but finitely many $P\in E(K)$, for every $n\ge 1$ with $(n,p)=1$ there exists a good place $v$ with $\mathrm{ord}_v({P})=n$.
\begin{proposition}
\label{prop:failure-all-but-finitely-many}
Let $p\ge 5$ and $K=\mathbb{F}_p(t)$.  There exists an elliptic curve $E/K$ and
infinitely many non-torsion points $P\in E(K)$ such that there is \emph{no} place
$v$ of good reduction of $E$ for which the reduction $\mathrm{red}_v({P})\in E(\kappa_v)$
is the identity (equivalently, has order $1$).

\end{proposition}

\begin{proof}
Let $E/K$ and $R\in E(K)$ be as in Lemma~\ref{lem:supersingular-isotrivial}, and
for each $k\ge 0$,  set $P_k:=p^kR$.

Writing $E:y^2=x^3+Ax+B$ with $A=\alpha r(t)^2$ and $B=\beta r(t)^3$, we have
\[
\Delta(E)=-16(4A^3+27B^2)=\Delta(E_0)\,r(t)^6,
\qquad \Delta(E_0)\in \mathbb{F}_p^\times.
\]
Hence every finite place $v$ such that $v(r(t)) = 0$ is a place of good reduction,
and the only finite places of bad reduction are those corresponding to the zeros of $r(t)$.

At $v_\infty$ we have ${v_\infty}(r(t))=-3$, so ${v_\infty}(\Delta(E))=-18$.
Under an admissible change of variables $(x,y)=(u^2x',u^3y')$, the discriminant
valuation changes by a multiple of $12$.  Since $-18\not\equiv 0\pmod{12}$, no
admissible change can make ${v_\infty}(\Delta)$ equal to $0$, hence $E$ has
\emph{bad} reduction at $v_\infty$.  

Let $v$ be a place of good reduction, then $v\neq v_\infty$.
By Lemma~\ref{lem:supersingular-isotrivial}, both affine coordinates of $P_k$ lie in $\mathcal{O}_v$,
so  $\mathrm{red}_v({P_k}) \neq\mathcal{O}$ by Lemma \ref{detecting torsion points}, hence $\mathrm{ord}_v({P_k})\neq1$. Since $R$ is non-torsion, $\{P_k:k\ge 0\}$ is an infinite set of non-torsion points with the stated property.
\end{proof}

\section{Unlikely intersection and a relative Silverman theorem}\label{sec:unlikely}

Throughout this section, we work over $\mathbb{C}$ and with complex analytifications, except in Section \ref{section: Zilber--Pink}.

\subsection{Uniformization and Betti coordinates}

Let $A/\mathbb{C}$ be an abelian variety of dimension $g$.
The analytic exponential map
\[
\exp_A : \operatorname{Lie}(A) \simeq \mathbb{C}^g \longrightarrow A(\mathbb{C})
\]
is a surjective holomorphic homomorphism with kernel a lattice
$\Lambda:=\ker(\exp_A)\simeq \mathbb{Z}^{2g}$.
Fix a $\mathbb{Z}$--basis $\omega_1,\dots,\omega_{2g}$ of $\Lambda$.
Viewing $\operatorname{Lie}(A)\simeq \mathbb{C}^g$ as a real vector space of dimension $2g$,
the vectors $\omega_1,\dots,\omega_{2g}$ form an $\mathbb{R}$--basis.
Hence every $z\in \operatorname{Lie}(A)$ can be written uniquely as
\[
z=\sum_{i=1}^{2g} b_i(z)\,\omega_i\qquad (b_i(z)\in \mathbb{R}).
\]
For $x\in A(\mathbb{C})$, we choose any $z\in \operatorname{Lie}(A)$ with $\exp_A(z)=x$, and define
\[
b_A(x):=\big(b_1(z),\dots,b_{2g}(z)\big)\bmod \mathbb{Z}^{2g}\ \in\
\mathbb{T}^{2g}:=\mathbb{R}^{2g}/\mathbb{Z}^{2g}.
\]
This is well-defined because replacing $z$ by $z+\lambda$ for $\lambda\in \Lambda$
adds an element of $\mathbb{Z}^{2g}$ to the coordinate vector.
The map $b_A:A(\mathbb{C})\to \mathbb{T}^{2g}$ is a real-analytic isomorphism of real Lie groups;
changing the $\mathbb{Z}$--basis of $\Lambda$ changes $b_A$ by an automorphism of $\mathbb{T}^{2g}$
induced by $\mathrm{GL}_{2g}(\mathbb{Z})$.

\begin{remark}\label{rem:torsion-betti}
A point $x\in A(\mathbb{C})$ is torsion if and only if $b_A(x)$ is a torsion point of $\mathbb{T}^{2g}$.
Moreover, $x$ has \emph{exact order} $N$ if and only if $b_A(x)$ has exact order $N$,
equivalently $b_A(x)$ admits a representative of the form
$\frac{1}{N}(a_1,\dots,a_{2g})$ with $a_i\in\mathbb{Z}$ and $\gcd(N,a_1,\dots,a_{2g})=1$.
\end{remark}

\subsection{Betti map}\label{subsec:betti-map}

We briefly recall the Betti map and refer to \cite[\S 3--\S 4]{Gao20} and \cite[\S 2.3]{GH23}
for detailed constructions and properties.

Let $\pi:\mathcal A\to S$ be an abelian scheme over a smooth irreducible quasi-projective
variety $S/\mathbb{C}$ of relative dimension $g\ge 1$, and let $P:S\to \mathcal A$ be a section.
We say that $P$ is \emph{non-torsion} if its image is not contained in $\mathcal A[N]$
for any $N\ge 1$ (equivalently, the generic point $p(\eta)$ is not torsion in
$\mathcal A_\eta(\overline{\mathbb{C}(S)})$).  We write $S^{an}$ for the \emph{analytification} of $S$.

Let $s\in S(\mathbb{C})$.  After shrinking $S^{an}$ around $s$, we may choose a simply connected open neighborhood $\Delta\subset S^{an}$ of $s$ such that the integral homology local system
\[
\mathcal L \ :=\ \big(H_1(\mathcal A_t(\mathbb{C}),\mathbb{Z})\big)_{t\in\Delta}
\qquad\text{(equivalently, $R^1\pi_\ast\mathbb{Z}$ up to duality)}
\]
is trivial on $\Delta$. Over $\Delta$, the abelian scheme $\mathcal A_\Delta:=\pi^{-1}(\Delta)$ admits an analytic
uniformization of the form
\[
\mathcal A_\Delta^{an}\ \simeq\ (\Delta\times \mathbb{C}^g) / \Lambda_\Delta,
\]
where $\Lambda_\Delta\simeq \mathbb{Z}^{2g}$ is a locally constant lattice in $\Delta\times \mathbb{C}^g$.
Fix a $\mathbb{Z}$--basis of ``periods''
\[
\omega_1(t),\dots,\omega_{2g}(t)\in \mathbb{C}^g\qquad (t\in \Delta),
\]
varying holomorphically with $t\in \Delta$, such that
$\Lambda_t=\mathbb{Z}\omega_1(t)\oplus\cdots\oplus \mathbb{Z}\omega_{2g}(t)$ is the period lattice of the fiber
$\mathcal A_t(\mathbb{C})$.

For $x\in \mathcal A_\Delta(\mathbb{C})$ lying above $t=\pi(x)\in \Delta$, choose a lift
$\tilde x\in \mathbb{C}^g$ of $x$ in the uniformization of $\mathcal A_t$, and write uniquely
\[
\tilde x=\sum_{i=1}^{2g} b_i(x)\,\omega_i(t)\qquad(b_i(x)\in \mathbb{R}).
\]
The \emph{Betti map} is the real-analytic map
\begin{align*}
b_\Delta:\mathcal A_\Delta(\mathbb{C}) &\longrightarrow \mathbb{T}^{2g}=\mathbb{R}^{2g}/\mathbb{Z}^{2g},\\
x &\longmapsto (b_1(x),\dots,b_{2g}(x))\bmod \mathbb{Z}^{2g}.
\end{align*}
It is well-defined up to composition with a real-analytic automorphism of $\mathbb{T}^{2g}$
coming from $\mathrm{GL}_{2g}(\mathbb{Z})$, and for each $t\in \Delta$ its restriction
$b_\Delta|_{\mathcal A_t(\mathbb{C})}:\mathcal A_t(\mathbb{C})\to \mathbb{T}^{2g}$ is a group isomorphism.

If $P:S\to\mathcal A$ is a section of $\pi$, we write $b_\Delta\circ P:\Delta\to \mathbb{T}^{2g}$
for the corresponding Betti coordinates of $p$ on $\Delta$.

\begin{definition}\label{def:betti-rank}
Let $X\subset \mathcal A$ be an irreducible closed subvariety with $\pi(X)=S$.
Define the \emph{generic Betti rank} of $X$ by
\[
\mathrm{rank}_{\mathbb{R}}(db_\Delta|_X)\ :=\
\max_{x\in X^{\mathrm{sm}}(\mathbb{C})\cap \mathcal A_\Delta}
\Big(\mathrm{rank}_{\mathbb{R}}\big(d(b_\Delta|_{X\cap \mathcal A_\Delta})_x\big)\Big)
\]
where $X^{\mathrm{sm}}$ is the smooth locus of $X$. This number does not depend on the choice of $\Delta$ (nor on the period basis), and satisfies the bound
\[
\mathrm{rank}_{\mathbb{R}}(db_\Delta|_X)\ \le\ 2\min(\dim X,\,g).
\]
We say that $b_\Delta|_X$ is \emph{generically submersive} if
$\mathrm{rank}_{\mathbb{R}}(db_\Delta|_X)=2g$.
\end{definition}

Assume $S$ is smooth and let $X=P(S)\subset \mathcal A$.
Since $\pi$ is separated, $p$ is a closed immersion, hence $X$ is smooth and
$X^{\mathrm{sm}}=X$. Moreover, for any simply connected $\Delta\subset S^{\mathrm{an}}$ on which
$b_\Delta$ is defined, we have $X\cap \mathcal A_\Delta = p(\Delta)$ and
$b_\Delta|_{X\cap \mathcal A_\Delta}=b_\Delta\circ p$ on $\Delta$.
Consequently,
\[
\mathrm{rank}_{\mathbb{R}}(db_\Delta|_{P(S)}) \ =\
\max_{s\in \Delta}\Big(\mathrm{rank}_{\mathbb{R}}\big(d(b_\Delta\circ p)_s\big)\Big).
\]

\begin{remark}\label{rem:torsion-fiber-betti}
For $t\in \Delta$ and $x\in \mathcal A_t(\mathbb{C})$, the point $x$ is $N$--torsion
(resp.\ of exact order $N$) in $\mathcal A_t(\mathbb{C})$ if and only if $b_\Delta(x)$
is $N$--torsion (resp.\ of exact order $N$) in $\mathbb{T}^{2g}$.
\end{remark}

\subsection{A relative Silverman theorem via the Betti map}

\begin{lemma}\label{lem:torsion-open}
Let $m\ge 2$ and let $U\subset \mathbb{T}^{m}$ be a non-empty open subset
(for the usual real manifold topology).
Then there exists an integer $N_0\ge 1$ such that for every integer $N\ge N_0$
there exists a point of \emph{exact order} $N$ contained in $U$.
\end{lemma}

\begin{proof}
Let $q:\mathbb{R}^{m}\to \mathbb{T}^{m}=\mathbb{R}^m/\mathbb{Z}^m$ be the quotient map.
Choose an open box $B=\prod_{i=1}^m(a_i,b_i)\subset \mathbb{R}^m$ such that $q(B)\subset U$
and $q|_B$ is injective (e.g.\ take $B$ contained in a fundamental domain).
Write $L_i:=b_i-a_i>0$ and set $c:=\prod_{i=1}^m L_i=\mathrm{vol}(B)>0$.

For $N\ge 1$ define
\[
S_N:=\Big\{k=(k_1,\dots,k_m)\in \mathbb{Z}^m:\ \frac{k}{N}\in B\Big\}.
\]
Then $q(k/N)\in U$ is $N$--torsion for each $k\in S_N$.
Moreover, $q(k/N)$ has \emph{exact} order $N$ if and only if $\gcd(N,k_1,\dots,k_m)=1$.

Set
\[
s_i(N):=\#\{k_i\in \mathbb{Z}:\ a_iN<k_i<b_iN\}.
\]
Then $\#S_N=\prod_{i=1}^m s_i(N)$.
For each $i$ we have $|s_i(N)-L_i N|\le 2$ for all $N\ge 1$.
Expanding the product gives a uniform estimate: there exists $C_B>0$ (depending only on $B$ and $m$) such that
\begin{equation}\label{eq:SN-uniform}
\big|\#S_N-cN^m\big|\le C_B N^{m-1}\qquad\text{for all }N\ge 1.
\end{equation}

Let $\kappa_m:=\sum_{\ell\ \mathrm{prime}}\ell^{-m}$.
Since $m\ge 2$, we have $\kappa_m\le \sum_{n\ge 2}n^{-m}=\zeta(m)-1<1$.

Assume for contradiction that for some $N$, every $k\in S_N$ satisfies $\gcd(N,k_1,\dots,k_m)>1$.
Then for each $k\in S_N$ there exists a prime divisor $\ell\mid N$ dividing every coordinate $k_i$.
Hence
\[
S_N\subseteq \bigcup_{\ell\mid N} S_{N,\ell},
\qquad
S_{N,\ell}:=\{k\in S_N:\ \ell\mid k_1,\dots,\ell\mid k_m\}.
\]
So
\begin{equation}\label{eq:union}
\#S_N\le \sum_{\ell\mid N}\#S_{N,\ell}.
\end{equation}
If $\ell\mid N$, the map $k\mapsto k/\ell$ induces a bijection $S_{N,\ell}\simeq S_{N/\ell}$, so $\#S_{N,\ell}=\#S_{N/\ell}$.
Using \eqref{eq:union} and \eqref{eq:SN-uniform}, we obtain
\[
\#S_N
\le \sum_{\ell\mid N}\#S_{N/\ell}
\le \sum_{\ell\mid N}\Big(c(N/\ell)^m + C_B (N/\ell)^{m-1}\Big)
\le c\kappa_m N^m + C_B N^{m-1}\sum_{\ell\mid N}1.
\]
Let $\nu(N)$ be the number of distinct prime divisors of $N$. Then $\sum_{\ell\mid N}1=\nu(N)$ and
$2^{\nu(N)}\le N$, hence $\nu(N)\le \log N/\log 2$.
Therefore
\[
\#S_N\le c\kappa_m N^m + C_B N^{m-1}\frac{\log N}{\log 2}.
\]
On the other hand, \eqref{eq:SN-uniform} gives $\#S_N\ge cN^m-C_BN^{m-1}$.
Thus
\[
c(1-\kappa_m)N^m \le C_B N^{m-1}\Big(1+\frac{\log N}{\log 2}\Big),
\]
which is impossible for $N$ sufficiently large since $1-\kappa_m>0$ and $m\ge 2$.
Hence for all large $N$ there exists $k\in S_N$ with $\gcd(N,k_1,\dots,k_m)=1$, and then $q(k/N)\in U$
has exact order $N$.
\end{proof}

We will apply Lemma~\ref{lem:torsion-open} in the case $m=2g\ge 2$ as follows.

\begin{theorem}\label{gen-sub}
Let $\pi:\mathcal A\to S$ be an abelian scheme over a smooth irreducible
quasi-projective variety $S/\mathbb{C}$ of relative dimension $g\ge 1$, and let
$P:S\to \mathcal A$ be a non-torsion section.
Assume that the Betti map restricted to $X:=P(S)$ is generically submersive, i.e.
\[
\mathrm{rank}_{\mathbb{R}}(db_\Delta|_{X})=2g.
\]
Then for every sufficiently large integer $N$ there exists $s\in S(\mathbb{C})$ such that
$P(s)$ has \emph{exact order} $N$ in the fiber $\mathcal A_s(\mathbb{C})$.
\end{theorem}

\begin{proof}
By the definition of the generic Betti rank (Definition~\ref{def:betti-rank}), the hypothesis
$\mathrm{rank}_{\mathbb{R}}(db_\Delta|_X)=2g$ means that there exist a point
$x_0\in X^{\mathrm{sm}}(\mathbb{C})$ and a simply connected open neighborhood
$\Delta\subset S^{\mathrm{an}}$ of $s_0:=\pi(x_0)$ such that $x_0\in \mathcal A_\Delta$ and
\[
\mathrm{rank}_{\mathbb{R}}\big(d(b_\Delta|_{X\cap\mathcal A_\Delta})_{x_0}\big)=2g.
\]
Since $\dim_{\mathbb{R}}(\mathbb{T}^{2g})=2g$, this differential is surjective, hence
$b_\Delta|_{X\cap\mathcal A_\Delta}$ is a submersion at $x_0$. Since $\pi$ is separated, the section $p$ is a (holomorphic) embedding, and $X\cap \mathcal A_\Delta=p(\Delta)$.
Equivalently, the real-analytic map $b_\Delta\circ P:\Delta\to \mathbb{T}^{2g}$ has surjective differential at $s_0$.

By the real submersion theorem, a submersion is locally open. Hence, there exists an open neighborhood $W\subset \Delta$ of $s_0$
and an open neighborhood $U\subset \mathbb{T}^{2g}$ of $b_\Delta(x_0)$ such that
\[
U\ \subset\ (b_\Delta\circ p)(W).
\]
Replacing $\Delta$ by $W$, we may and do assume that $U\subset (b_\Delta\circ p)(\Delta)$. Since $2g\ge 2$, Lemma \ref{lem:torsion-open} yields an integer $N_0$ such that for each $N\ge N_0$ we can choose $u_N\in U$ of exact order $N$ in $\mathbb{T}^{2g}$.
For such $N$ pick $s_N\in \Delta$ with
\[
b_\Delta(p(s_N))=u_N.
\]
By Remark \ref{rem:torsion-fiber-betti}, $p(s_N)$ has exact order $N$ in the fiber
$\mathcal A_{s_N}(\mathbb{C})$.
\end{proof}

If $\mathrm{rank}_{\mathbb{R}}(db_\Delta|_{P(S)})=2g$, then $p$ is automatically non-torsion: indeed, if $p$ factored through $\mathcal A[N]$ for some $N$, then $p(\Delta)$ would be locally constant
(since $\mathcal A[N]\to S$ is finite \'etale), and hence $b_\Delta\circ p$ could not be a submersion.
Moreover, the bound in Definition~\ref{def:betti-rank} shows that generic submersivity forces
$\dim S\ge g$.

\subsection{A criterion via the relative Manin--Mumford conjecture}

Let $\mathcal A_{\mathrm{tor}}\subset \mathcal A(\mathbb{C})$ denote the union over all
$s\in S(\mathbb{C})$ of the torsion subgroup $\mathcal A_s(\mathbb{C})_{\mathrm{tor}}$.

\begin{definition}\label{def:ZX}
For $N\in\mathbb Z$, let $[N]:\mathcal A\to\mathcal A$ denote the multiplication-by-$N$ morphism.
For a subvariety $X\subset\mathcal A$ we set
\[
\mathbb{Z}X \ :=\ \bigcup_{N\in\mathbb{Z}} [N](X).
\]
We say that $\mathbb{Z}X$ is \emph{Zariski dense} in $\mathcal A$ if its Zariski closure equals $\mathcal A$.
\end{definition}

\begin{corollary}\label{cor:GH23-to-silverman}
Keep the notation of Theorem~\ref{gen-sub} and set $X:=P(S)\subset \mathcal A$.
Assume that $\mathbb{Z} X$ is Zariski dense in $\mathcal A$ and that
$X(\mathbb{C})\cap \mathcal A_{\mathrm{tor}}$ is Zariski dense in $X$.
Then the conclusion of Theorem~\ref{gen-sub} holds.
\end{corollary}

\begin{proof}
By \cite[Theorem~1.3]{GH23}, under the hypothesis that $\mathbb{Z} X$ is Zariski dense in $\mathcal A$,
the Zariski denseness of $X(\mathbb{C})\cap \mathcal A_{\mathrm{tor}}$ in $X$ is equivalent to the maximal
Betti rank $\mathrm{rank}_{\mathbb{R}}(db_\Delta|_X)$ $=2g$.
Hence $b_\Delta|_X$ is generically submersive and Theorem~\ref{gen-sub} applies.
\end{proof}

\begin{remark}\label{rem:dimS}
If $\dim S<g$ and $\mathbb{Z} X$ is Zariski dense in $\mathcal A$, then
$X(\mathbb{C})\cap \mathcal A_{\mathrm{tor}}$ is \emph{not} Zariski dense in $X$
by \cite[Theorem~1.1]{GH23}.
In particular, if $S$ is a curve and $g\ge 2$, then $X(\mathbb{C})\cap \mathcal A_{\mathrm{tor}}$
is finite.
\end{remark}

\subsection{Curves and a Zilber--Pink type characterization}\label{section: Zilber--Pink}

We end with a standard obstruction for curves in a \emph{fixed} abelian variety,
which is a strong case of the Zilber--Pink philosophy proved by Barroero--Dill \cite{BD22}.

Let $A$ be an abelian variety over an algebraically closed field of characteristic $0$.
A \emph{special subvariety} of $A$ means an irreducible component of an algebraic subgroup
of $A$ (equivalently, a translate of an abelian subvariety by a torsion point).
For $k\ge 0$, let $A^{[k]}$ denote the union of all special subvarieties of $A$
of codimension at least $k$.
Note that $A^{[g]}=A_{\mathrm{tor}}$, since special subvarieties of codimension $\ge g$
are precisely $0$-dimensional special subvarieties, i.e.\ torsion points.

\begin{theorem}[Barroero--Dill \cite{BD22}]\label{thm:BD}
Let $A$ be an abelian variety over an algebraically closed field of characteristic $0$
and let $V\subset A$ be a curve. Then $V\cap A^{[2]}$ is finite unless $V$ is contained
in a proper algebraic subgroup of $A$.
\end{theorem}

\begin{corollary}\label{Zilber--Pink}
Let $A$ be an abelian variety over an algebraically closed field of characteristic $0$
of dimension $g\ge 2$, and let $V\subset A$ be an irreducible curve.
Assume that for every sufficiently large integer $N$ there exists a point $x_N\in V$
of exact order $N$ in $A$.
Then $V$ is contained in a proper algebraic subgroup of $A$.
\end{corollary}

\begin{proof}
The assumption implies that $V\cap A_{\mathrm{tor}}$ is infinite.
Since $g\ge 2$, every torsion point has codimension $g\ge 2$, hence
$A_{\mathrm{tor}}=A^{[g]}\subseteq A^{[2]}$.
Therefore $V\cap A^{[2]}$ is infinite, and Theorem \ref{thm:BD} forces
$V$ to be contained in a proper algebraic subgroup of $A$.
\end{proof}

\appendix

\section*{Appendix A. Schinzel's theorem for global function fields}
In this appendix, we prove a version of Schinzel's theorem for tori over global function fields. For multiplicative groups over rational function fields, we refer to \cite[ Theorem~1.3]{Fla10}. For the sake of the convenience of the readers, we present a different proof. The characteristic $p$ in this appendix is not necessarily different from 2 or 3.
\begin{manualproposition}{A.1}\label{multiplicative group}
Let $K$ be a global function field with constant field $\mathbb{F}_q$ of characteristic $p$,
and let $x\in K^\times$ be not a root of unity. Then for every integer $n>1$ coprime to $p$, there exists a place $v\in M_K$ with $v(x)=0$ such that the order of the reduction of $x$ in $\kappa_v^\times$ equals $n$. 
\end{manualproposition}
\begin{manualremark}{A.1}
    Again, the condition $n$ being coprime to $p$ is necessary. Indeed, for any $x\in K$, we have $x^{np}-1=(x^n-1)^p$. Therefore, if $v(x^{np}-1)>0$ for some place $v$, then $v(x^{n}-1)>0$.
\end{manualremark}
\begin{proof}We denote by $\mu$ the M\"obius function, by $\Phi_n$ the $n$-th cyclotomic polynomial, and by $\phi$ the Euler totient function. We set $W:=\{v\in M_K:v(x^n-1)>0\}.$ We need to find $v$ such that $v(x)\geq0$ and $\mathrm{ord}_v(x)=n$. We consider the following four cases. 
\begin{enumerate}
    \item $v\in W$ such that $\mathrm{ord}_v(x)\neq n$, we call this order by $n_0.$   Then $n=n_0k$ for some positive integer $k$ and
 $$x^n-1=(x^{n_0}-1)(x^{n_0(k-1)}+x^{n_0(k-2)}+...+x^{n_0}+1). $$
 Since $x^{n_0(k-1)}+x^{n_0(k-2)}+...+x^{n_0}+1\equiv k$  and $(k,p)=1$, we have $v(x^n-1)=v(x^{n_0}-1)$. Thus 
\begin{equation} 
\begin{split}
v(\Phi_n(x)) & =\sum_{m|n}\mu\Big(\dfrac{n}{m}\Big)v(x^m-1)   = \sum_{n_0|m|n}\mu\Big(\dfrac{n}{m}\Big)v(x^m-1)\\
 & = \sum_{n_0|m|n}\mu\Big(\dfrac{n}{m}\Big)v(x^{n_0}-1)
=0\text{ since }n_0<n.
\end{split}
\end{equation}
\item $v\in M_K$ satisfying $v(x)>0$. Then $v(x^m-1)=0$ for all positive integers $m$. It implies that $v(\Phi_n(x))=0.$
\item $v\in M_K$ satisfying $v(x)<0$. Then $v(x^m-1)=v(1-x^{-m})+v(x^m)=mv(x).$ Hence 
\begin{align*}
v(\Phi_n(x)) =\sum_{m|n}\mu\Big(\dfrac{n}{m}\Big)v(x^m-1)   = \sum_{m|n}\mu\Big(\dfrac{n}{m}\Big)mv(x)= \phi(n)v(x).
\end{align*}
\item $v\in M_K$ satisfying $v(x)=0$ and $v\not\in W$. Then $v(x^m-1)=0$ for all $m\mid n$,
and hence $v(\Phi_n(x))=0$.
\end{enumerate}
Combining the four cases and using the product formula for the nonzero element $\Phi_n(x)\in K^\times$, we get
\[
0=\sum_{v\in M_K} v(\Phi_n(x))\deg(v)
=\phi(n)\sum_{\substack{v\in M_K\\ v(x)<0}} v(x)\deg(v).
\]
Each term in the last sum is non-positive, and it is strictly negative whenever $v(x)<0$.
Hence the sum can vanish only if $v(x)\ge 0$ for all $v$.
Thus $x$ has no pole on the smooth projective curve with function field $K$, so $x\in \mathbb F_q^\times$.
But every nonzero element of $\mathbb F_q$ is a root of unity, contradicting the hypothesis on $x$.
This contradiction proves the proposition.
\end{proof}

\begin{manualtheorem}{A.1}\label{Torus}
Let $G$ be a one-dimensional torus over a global function field $K$, and let $P\in G(K)$ be a non-torsion point. Then for every sufficiently large integer $n$ coprime to $p$, there exists a place $v\in M_K$ such that the reduction of $P$ at $v$ exists and
\[
\mathrm{ord}_v(P)=n.
\]
\end{manualtheorem}

\begin{proof}
Choose a finite set of places $S\subset M_K$ such that $G$ extends to a smooth affine group scheme
\[
\mathcal G\longrightarrow \Spec(\mathcal O_{K,S})
\]
and such that $P\in \mathcal G(\mathcal O_{K,S})$.
Let $K'/K$ be a finite Galois extension that splits $G$, and let $S'$ be the set of places of $K'$ lying over $S$.
Choose an isomorphism of $K'$-groups
\[
\psi:G_{K'}\xrightarrow{\sim}\mathbb G_{m,K'}.
\]
After enlarging $S$ (and hence $S'$) if necessary, this isomorphism spreads out to an isomorphism of group schemes
\[
\psi:\mathcal G\times_{\Spec\mathcal O_{K,S}}\Spec\mathcal O_{K',S'}
\xrightarrow{\sim}
\mathbb G_{m,\mathcal O_{K',S'}}.
\]
Let
\[
x':=\psi(P_{K'})\in \mathbb G_m(K')=K'^\times.
\]
Since $P$ is non-torsion and $\psi$ is an isomorphism on the generic fiber, $x'$ is not a root of unity.

Applying Proposition~\ref{multiplicative group} over $K'$, we obtain an integer $N_0$ such that for every integer $n>N_0$ with $\gcd(n,p)=1$, there exists a place $w\in M_{K'}\setminus S'$ satisfying
\[
\mathrm{ord}_w(x')=n
\]
in $\kappa_w^\times$.
Let $v\in M_K\setminus S$ be the place below $w$.
Reducing the isomorphism $\psi$ modulo $w$ yields an isomorphism of special fibers
\[
\psi_w:\mathcal G_{\kappa_w}\xrightarrow{\sim}\mathbb G_{m,\kappa_w}.
\]
The reduction of $P$ at $v$, viewed in $\mathcal G(\kappa_w)$ via the natural injective map
\[
\mathcal G(\kappa_v)\hookrightarrow \mathcal G(\kappa_w),
\]
coincides with the reduction of $P_{K'}$ at $w$; under $\psi_w$ this element maps to $x'\bmod w$, which has order $n$.
Therefore the reduction of $P$ at $v$ already has order $n$ in $\mathcal G(\kappa_v)$.
Hence $\mathrm{ord}_v(P)=n$.
\end{proof}

\bibliographystyle{plain} 
\bibliography{references} 

\vspace{1cm}
\author{Khai-Hoan Nguyen-Dang}\\
\email{Khaihoann@gmail.com}\\
\address{Morningside Center of Mathematics, Chinese Academy of Sciences, Beijing, China}\\
 
\author{Quang-Khai Nguyen} \\
\email{nguyen@math.univ-lyon1.fr} \\
\address{Camille Jordan Institute, Claude Bernard University Lyon 1\\ 21 avenue Claude Bernard, 69100 Villeurbanne, France}
\end{document}